\documentclass[a4paper,superscriptaddress,11pt, unpublished, shorttitle=template]{compositionalityarticle}
\input{idrisLang.sty}

\usepackage{graphicx}
\usepackage{amsthm}
\usepackage{mathtools} 
\usepackage{amssymb} 

\usepackage{subcaption} 

\usepackage{tikz}
\usetikzlibrary{
	cd,
	math,
	decorations.markings,
	positioning,
	arrows,
	arrows.meta,
	shapes,
	calc,
	fit,
	quotes,
	automata,
	petri,
	scopes,
	backgrounds,
	quotes
}

\usepackage[
	style=numeric, 
	maxnames=3,
	bibencoding=utf8,
	backend=bibtex,
	defernumbers=true,
	backref=true,
	hyperref=true,
	url=false,
	isbn=false,
	alldates=long
]{biblatex}

\makeatletter
\let\blx@rerun@biber\relax
\makeatother

\AtEveryBibitem{
	\clearlist{address}
	\clearlist{location}
	\clearfield{series}
}
\def\backgrnd{black!10}	

\tikzstyle{place}=
[circle,thick,draw=blue!75,fill=blue!20,minimum size=6mm]
\tikzstyle{transition}=
[rectangle,thick,draw=black!75,fill=black!20,minimum size=4mm]

\tikzset{
	oriented WD/.style={
		every to/.style={out=0,in=180,draw},
		label/.style={
			font=\everymath\expandafter{\the\everymath\scriptstyle},
			inner sep=0pt,
			node distance=2pt and -2pt},
		semithick,
		node distance=1 and 1,
		decoration={markings, mark=at position \stringdecpos with \stringdec},
		ar/.style={postaction={decorate}},
		execute at begin picture={\tikzset{
				x=\bbx, y=\bby,
				every fit/.style={inner xsep=\bbx, inner ysep=\bby}}}
	},
	string decoration/.store in=\stringdec,
	string decoration={\arrow{stealth};},
	string decoration pos/.store in=\stringdecpos,
	string decoration pos=.7,
	bbx/.store in=\bbx,
	bbx = 1.5cm,
	bby/.store in=\bby,
	bby = 1.5ex,
	bb port sep/.store in=\bbportsep,
	bb port sep=1.5,
	bb port length/.store in=\bbportlen,
	bb port length=4pt,
	bb penetrate/.store in=\bbpenetrate,
	bb penetrate=0,
	bb min width/.store in=\bbminwidth,
	bb min width=1cm,
	bb rounded corners/.store in=\bbcorners,
	bb rounded corners=2pt,
	bb small/.style={bb port sep=1, bb port length=2.5pt, bbx=.4cm, bb min width=.4cm, 
		bby=.7ex},
	bb medium/.style={bb port sep=1, bb port length=2.5pt, bbx=.4cm, bb min width=.4cm, 
		bby=.9ex},
	bb/.code 2 args={
		\pgfmathsetlengthmacro{\bbheight}{\bbportsep * (max(#1,#2)+1) * \bby}
		\pgfkeysalso{draw,minimum height=\bbheight,minimum width=\bbminwidth,outer 
			sep=0pt,
			rounded corners=\bbcorners,thick,
			prefix after command={\pgfextra{\let\fixname\tikzlastnode}},
			append after command={\pgfextra{\draw
					\ifnum #1=0{} \else foreach \i in {1,...,#1} {
						($(\fixname.north west)!{\i/(#1+1)}!(\fixname.south west)$) +(-
						\bbportlen,0) 
						coordinate (\fixname_in\i) -- +(\bbpenetrate,0) coordinate (\fixname_in\i')}\fi 
					\ifnum #2=0{} \else foreach \i in {1,...,#2} {
						($(\fixname.north east)!{\i/(#2+1)}!(\fixname.south east)$) +(-
						\bbpenetrate,0) 
						coordinate (\fixname_out\i') -- +(\bbportlen,0) coordinate (\fixname_out\i)}\fi;
		}}}
	},
	bb name/.style={append after command={\pgfextra{\node[anchor=north] at 
				(\fixname.north) {#1};}}}
}

\tikzset{
	unoriented WD/.style={
		every to/.style={draw},
		shorten <=-\penetration, shorten >=-\penetration,
		label distance=-2pt,
		thick,
		node distance=\spacing,
		execute at begin picture={\tikzset{
				x=\spacing, y=\spacing}}
	},
	pack size/.store in=\psize,
	pack size = 8pt,
	spacing/.store in=\spacing,
	spacing = 8pt,
	link size/.store in=\lsize,
	link size = 2pt,
	penetration/.store in=\penetration,
	penetration = 2pt,
	pack color/.store in=\pcolor,
	pack color = blue,
	pack inside color/.store in=\picolor,
	pack inside color=blue!20,
	pack outside color/.store in=\pocolor,
	pack outside color=blue!50!black,
	surround sep/.store in=\ssep,
	surround sep=8pt,
	link/.style={
		circle, 
		draw=black, 
		fill=black,
		inner sep=0pt, 
		minimum size=\lsize
	},
	pack/.style={
		circle, 
		draw = \pocolor, 
		fill = \picolor,
		inner sep = .25*\psize,
		minimum size = \psize
	},
	outer pack/.style={
		ellipse, 
		draw,
		inner sep=\ssep,
		color=\pocolor,
	},
	intermediate pack/.style={
		ellipse,
		dashed, 
		draw,
		inner sep=\ssep,
		color=\pocolor,
	},
}

\tikzset{Yonepart/.pic={
		\node[bb={1}{2},bb name = {\tiny$X_{11}$}] (X11) {};
		\node[bb={2}{2},below right=of X11,bb name = {\tiny$X_{12}$}] (X12) {};
		\node[bb={2}{1}, above right=of X12,bb name = {\tiny$X_{13}$}] (X13) {};
		\node[bb={2}{2}, fit={($(X11.north west)+(.3,1.5)$) (X12)  ($(X13.east)+(-.3,0)$)},bb name = {\scriptsize $Y_1$}] (Y1) {};
		\draw (Y1_in1') to (X11_in1);	
		\draw (Y1_in2') to (X12_in2);
		\draw (X11_out1) to (X13_in1);
		\draw (X11_out2) to (X12_in1);
		\draw (X12_out1) to (X13_in2);
		\draw (X12_out2) to (Y1_out2');
		\draw (X13_out1) to (Y1_out1');
		\coordinate (bottombox) at ($(X12.south)$);
		\coordinate (rightbox) at ($(X13.east)$);
		\coordinate (Y1northwest) at ($(Y1.north west)$);
	}
}

\tikzset{Ytwopart/.pic={
		\node[bb={2}{2}, bb name = {\tiny$X_{21}$}] (X21) {};
		\node[bb={1}{2},above right=-1 and 1 of X21,bb name = {\tiny$X_{22}$}] (X22) {};
		\node[bb={1}{2}, fit={($(X21.south west)+(-.25,0)$) ($(X22.north east)+(.25,3.5)$)},bb name = {\scriptsize$Y_2$}] (Y2){};
		\draw (Y2_in1') to (X21_in2);
		\draw (X21_out1) to (X22_in1);
		\draw (X22_out2) to (Y2_out1');
		\draw let \p1=(X22.south east), \p2=($(Y2_out2)$), \n1={\y1-\bby}, \n2=\bbportlen in
		(X21_out2) to (\x1+\n2,\n1) -- (\x1+\n2,\n1) to (Y2_out2');
		\draw let \p1=(X22.north east), \p2=(X21.north west), \n1={\y1+\bby}, \n2=\bbportlen in
		(X22_out1) to[in=0] (\x1+\n2,\n1) -- (\x2-\n2,\n1) to[out=180] (X21_in1);
	}
}

\tikzset{SmallNeuronPic/.pic={
		\node[bb={3}{1}] (N1) {$\scriptstyle N_1$};
		\node[bb={2}{1}, above right=.7 and 3.5 of N1] (N2) {$\scriptstyle N_2$};
		\node[bb={2}{1}, below =of N2] (N3) {$\scriptstyle N_3$};
		\node[bb={3}{1}, below =of N3] (N4) {$\scriptstyle N_4$};
		\node[bb={6}{8}, fit={($(N1.west)-(.5,0)$) ($(N2.north)+(0,2)$) ($(N3.east)+(1.5,0)$) ($(N4.south)-(0,1)$)}, bb name={$\scriptstyle X$}] (X) {};
		\draw (X_in1') to (N2_in1);
		\draw (X_in2') to (N1_in1);
		\draw (X_in3') to (N1_in2);
		\draw (X_in4') to (N1_in3);
		\draw (X_in6') to (N4_in2);
		\draw (N1_out1) to (N2_in2);
		\draw (N1_out1) to (N3_in1);
		\draw (N1_out1) to (N4_in1);
		\draw (N2_out1) to (X_out1');
		\draw (N2_out1) to (X_out2');
		\draw (N2_out1) to (X_out3');
		\draw (N3_out1) to (X_out4');
		\draw (N3_out1) to (X_out5');
		\draw (N3_out1) to (X_out6');
		\draw (N4_out1) to (X_out7');
		\draw (N4_out1) to (X_out8'); 
		\draw (X_in5') to[looseness=2] (N3_in2);
		\draw let \p1=(N4.south east), \p2=(N4.south west), \n1={\y2-\bby}, \n2=\bbportlen in
		(N3_out1) to[in=0] (\x1+\n2,\n1) -- (\x2-\n2,\n1) to[out=180] (N4_in3);
	}
}

\tikzset{SmallNeuronDashed/.pic={
		\node[bb={3}{1}] (N1) {$\scriptstyle N_1$};
		\node[bb={2}{1}, above right=.7 and 3.5 of N1] (N2) {$\scriptstyle N_2$};
		\node[bb={2}{1}, below =of N2] (N3) {$\scriptstyle N_3$};
		\node[bb={3}{1}, below =of N3] (N4) {$\scriptstyle N_4$};
		\node[bb={6}{8}, fit={($(N1.west)-(.5,0)$) ($(N2.north)+(0,2)$) ($(N3.east)+(1.5,0)$) ($(N4.south)-(0,1)$)}, bb name={$\scriptstyle X$}] (X) {};
		\draw[dashed] (X_in1') to (N2_in1);
		\draw[dashed] (X_in2') to (N1_in1);
		\draw[dashed] (X_in3') to (N1_in2);
		\draw[dashed] (X_in4') to (N1_in3);
		\draw[dashed] (X_in6') to (N4_in2);
		\draw[dashed] (N1_out1) to (N2_in2);
		\draw[dashed] (N1_out1) to (N3_in1);
		\draw[dashed] (N1_out1) to (N4_in1);
		\draw[dashed] (N2_out1) to (X_out1');
		\draw[dashed] (N2_out1) to (X_out2');
		\draw[dashed] (N2_out1) to (X_out3');
		\draw[dashed] (N3_out1) to (X_out4');
		\draw[dashed] (N3_out1) to (X_out5');
		\draw[dashed] (N3_out1) to (X_out6');
		\draw[dashed] (N4_out1) to (X_out7');
		\draw[dashed] (N4_out1) to (X_out8'); 
		\draw[dashed] (X_in5') to[looseness=2] (N3_in2);
		\draw[dashed] let \p1=(N4.south east), \p2=(N4.south west), \n1={\y2-\bby}, \n2=\bbportlen in
		(N3_out1) to[in=0] (\x1+\n2,\n1) -- (\x2-\n2,\n1) to[out=180] (N4_in3);
	}
}

\tikzset{SmallNestingPic/.pic={
		\path (0,0) pic [purple] {Yonepart};
		\path ($(rightbox)+(5,-5)$) pic [orange] {Ytwopart};
		
		\node[bb={1}{2}, fit={($(Y1northwest)+(-.5,4)$) ($(Y2.south east)+(1,0)$)}, bb name={\small $Z$}] (Z) {};
		\draw (Z_in1') to (Y1_in2);
		\draw let \p1=(Y2.north west),\p2=(Y2.north east),\n1={\y2+\bby},\n2=\bbportlen in
		(Y1_out1) to (\x1+\n2,\n1)--(\x2+\n2,\n1) to (Z_out1');
		\draw (Y1_out2) to (Y2_in1);
		\draw (Y2_out2) to (Z_out2');
		\draw let \p1=(Y2.north east), \p2=(Y1.north west), \n1={\y2+\bby}, \n2=\bbportlen in
		(Y2_out1) to[in=0] (\x1+\n2,\n1) -- (\x2-\n2,\n1) to[out=180] (Y1_in1);
	}
}

\tikzset{Zredgreen/.pic={
		\node[bb={2}{2}, green!50!black, bb name = $\scriptstyle Y_1$] (YY1) {};
		\node[bb={1}{2}, red, below right=-1 and 2 of YY1, bb name=$\scriptstyle Y_2$] (YY2) {};
		\node[bb={1}{2}, fit={($(YY1.north west)+(-.5,4)$) ($(YY2.south east)+(.5,-2)$)}, bb name={\scriptsize $Z$}] (Z) {};
		\draw (Z_in1') to (YY1_in2);
		\draw (YY1_out1) to (Z_out1');
		\draw (YY1_out2) to (YY2_in1);
		\draw (YY2_out2) to (Z_out2');
		\draw let \p1=(YY2.north east), \p2=(YY1.north west), \n1={\y2+\bby}, \n2=\bbportlen in
		(YY2_out1) to[in=0] (\x1+\n2,\n1) -- (\x2-\n2,\n1) to[out=180] (YY1_in1);
	}
}

\tikzset{Zcombined/.pic={
		\node[bb={1}{2},green!25!black,bb name = {\tiny$X_{11}$}] (X11) {};
		\node[bb={2}{2},green!25!black,below right=of X11,bb name = {\tiny$X_{12}$}] (X12) {};
		\node[bb={2}{1}, green!25!black,above right=of X12,bb name = {\tiny$X_{13}$}] (X13) {};
		\draw (X11_out1) to (X13_in1);
		\draw (X11_out2) to (X12_in1);
		\draw (X12_out1) to (X13_in2);
		
		\node[bb={2}{2}, red!30!black, below right = 0 and 1.25 of X12, bb name = {\tiny$X_{21}$}] (X21) {};
		\node[bb={1}{2}, red!30!black, above right=-1 and 1 of X21,bb name = {\tiny$X_{22}$}] (X22) {};
		\draw (X21_out1) to (X22_in1);
		\draw let \p1=(X22.north east), \p2=(X21.north west), \n1={\y1+\bby}, \n2=\bbportlen in
		(X22_out1) to[in=0] (\x1+\n2,\n1) -- (\x2-\n2,\n1) to[out=180] (X21_in1);
		
		\node[bb={1}{2}, fit = {($(X11.north east)+(-1,3)$) (X12) (X13) ($(X21.south)+(0,-1)$) ($(X22.east)+(.5,0)$)}, bb name ={\scriptsize $Z$}] (Z) {};
		
		\draw (Z_in1') to (X12_in2);
		\draw (X13_out1) to (Z_out1');
		\draw (X12_out2) to (X21_in2);
		\draw let \p1=(X22.south east),\n1={\y1-\bby}, \n2=\bbportlen in
		(X21_out2) to (\x1+\n2,\n1) to (Z_out2');
		\draw let \p1=(X22.north east), \p2=(X11.north west), \n1={\y2+\bby}, \n2=\bbportlen in
		(X22_out2) to[in=0] (\x1+\n2,\n1) -- (\x2-\n2,\n1) to[out=180] (X11_in1);
	}
}

	\tikzset{
	pics/netA/.style args={#1/#2/#3/#4/#5/#6/#7}{code={	
			
			\node [place,label=above:$p_1$, tokens={%
				#1
			}] (-pl_1) {};
			
			\node [transition,label=above:$t$, label=below:#5] (-tr_1) [right = of -pl_1] {};
			
			\node [place,label=above:$p_2$, tokens={%
				#2
			}] (-pl_2) [right = of -tr_1] {};
			
			\node [transition,label=left:$v$, label=above:#6] (-tr_2) [below = of -tr_1] {};
			\node [transition,label=below:$u$, label=above:#7] (-tr_3) [below = of -tr_2] {};
			
			\node [place,label=below:$p_3$, tokens={%
				#3
			}] (-pl_3) [left = of -tr_3] {};
			
			\node [place,label=below:$p_4$, tokens={%
				#4
			}] (-pl_4) [right = of -tr_3] {};
			
			\draw[->] (-pl_1) -- (-tr_1);
			\draw[->] (-tr_1) -- (-pl_2);
			\draw[->] (-pl_2) -- (-tr_2);
			\draw[->] (-tr_2) -- (-pl_3);
			\draw[->] (-tr_2) -- (-pl_4);
			\draw[->] (-pl_3) -- (-tr_3);
			\draw[->] (-tr_3) -- (-pl_4);			
	}}
}

\usepackage{hyperref} 

\newtheorem{theorem}{Theorem}
	\AfterEndEnvironment{theorem}{\noindent\ignorespaces}
\newtheorem{proposition}{Proposition}
	\AfterEndEnvironment{proposition}{\noindent\ignorespaces}
\newtheorem{definition}{Definition}
	\AfterEndEnvironment{definition}{\noindent\ignorespaces}
\newtheorem{fact}{Fact}
	\AfterEndEnvironment{fact}{\noindent\ignorespaces}
\newtheorem{example}{Example}
	\AfterEndEnvironment{example}{\noindent\ignorespaces}
	\AfterEndEnvironment{proof}{\noindent\ignorespaces}

\newcommand{\Pl}[1]{P_{#1}} 
\newcommand{\Tr}[1]{T_{#1}} 
\newcommand{\Pin}[2]{{^\circ}(#1)_{#2}} 
\newcommand{\Pout}[2]{{(#1)_{#2}^\circ}} 
\newcommand{\Net}[1]{(\Pl{#1},\Tr{#1},\Pin{-}{#1},\Pout{-}{#1})} 

\newcommand{\Strings}[1]{#1^{\otimes}} 

\newcommand{\naturals}{\mathbb{N}}

\newcommand{\CategoryC}{\mathcal{C}}
\newcommand{\CategoryD}{\mathcal{D}}
\newcommand{\CategoryE}{\mathcal{E}}

\newcommand{\Petri}{\textbf{Petri}}
\newcommand{\tensor}{\otimes}
\newcommand{\tensorUnit}{\epsilon}
\newcommand{\id}[1]{id_{#1}} 

\newcommand{\suchthat}[2]{\left\{#1 \: \middle\vert \: #2\right\}} 

\newcommand{\PlM}[1]{{#1}_{Pl}} 
\newcommand{\TrM}[1]{{#1}_{Tr}} 

\newcommand{\Sym}[1]{\operatorname{Sym}_{#1}} 

\newcommand{\Mset}[1]{{#1}^{\oplus}} 
\newcommand{\Msets}[1]{#1^{\oplus}} 

\newcommand{\MultiplicitySym}{\mathfrak{M}} 
\newcommand{\Multiplicity}[1]{\MultiplicitySym_{#1}} 

\newcommand{\OrderingSym}{\mathfrak{O}} 
\newcommand{\Ordering}[1]{\OrderingSym_{#1}} 

\newcommand{\GPetri}{\Petri^{G}}
\newcommand{\FOSSMC}{\textbf{FOSSMC}}
\newcommand{\FSSMC}{\textbf{FSSMC}}
\newcommand{\GFOSSMC}{\FOSSMC^{G}}
\newcommand{\GFSSMC}{\FSSMC^{G}}

\newcommand{\FoldSym}{\mathcal{Q}}
\newcommand{\UnFoldSym}{\mathcal{N}}
\newcommand{\OFoldSym}{\mathcal{F}}

\newcommand{\OPetri}{\textbf{Petri}_{<}}
\newcommand{\OFSSMC}{\textbf{FSSMC}_{<}}

\newcommand{\idrisct}{{\texttt{idris-ct}}\xspace}
\begin{document}
\title{Computational Petri Nets: Adjunctions Considered Harmful}
%
%
%
%
\author{Fabrizio Genovese}	
	\orcid{0000-0001-7792-1375}
	\email{fabrizio@statebox.io}
\author{Alex Gryzlov}
	\orcid{0000-0001-6188-0417}
	\email{alex@statebox.io}
\author{Jelle Herold}
	\orcid{0000-0002-1966-2536}
	\email{jelle@statebox.io}
\author{Marco Perone}
	\orcid{0000-0002-1004-0431}
	\email{marcosh@statebox.io}
\author{Erik Post}
	\orcid{0000-0002-8111-9593}
	\email{erik@statebox.io}
\author{Andr\'e Videla}
	\orcid{0000-0002-7298-6230}
	\email{andre@statebox.io}
\affiliation{Statebox Team\\
\url{https://statebox.org}}
\maketitle
\begin{abstract}
	\noindent
We review some of the endeavors in trying to connect Petri nets with free symmetric monoidal categories. We give a list of requirements such connections should respect if they are meant to be useful for practical/implementation purposes. We show how previous approaches do not satisfy them, and give compelling evidence that this depends on trying to make the correspondence functorial in the direction from nets to free symmetric monoidal categories, in order to produce an adjunction. We show that dropping this immediately honors our desiderata, and conclude by introducing an Idris library which implements them.
\end{abstract}
\section{Introduction and motivation}
Among experts in concurrency and category theory, Petri nets have always been informally regarded as presentations of free strict symmetric monoidal categories (FSSMCs). The intuition behind such claims is simple: Given a Petri net, we can use its set of places to generate the objects of a FSSMC. Similarly, each of its transitions is considered as a generating morphism in the corresponding category, with domain and codomain the monoidal product of its input/output places, respectively. In this setting, a marking of the net is just an object in the corresponding FSSMC, and any sequence of transition firings can be mapped to a morphism. An example of this can be seen in Figure~\ref{fig: execution alongside net}, where we made use of wiring diagrams~\cite{Selinger2010} to represent FSSMCs morphisms.
\begin{figure}[h!]
	\centering
	\scalebox{0.5}{
		\begin{tikzpicture}
			\pgfmathsetmacro\bS{5}
			\pgfmathsetmacro\hkX{(\bS/3.5)}
			\pgfmathsetmacro\kY{-1.5}
			\pgfmathsetmacro\hkY{\kY*0.5}
			
			\draw pic (m0) at (0,0) {netA={{1}/{1}/{2}/{0}/{}/{}/{}}};
			\draw pic (m1) at (\bS,0) {netA={{0}/{2}/{2}/{0}/{$\blacktriangle$}/{}/{}}};
			\draw pic (m2) at ({2 * \bS},0) {netA={{0}/{1}/{3}/{1}/{}/{$\blacktriangledown$}/{}}};
			\draw pic (m3) at ({3 * \bS},0) {netA={{0}/{1}/{2}/{2}/{}/{}/{$\blacktriangledown$}}};
			
			\begin{scope}[very thin]
				\foreach \j in {1,...,3} {
					\pgfmathsetmacro \k { \j * \bS - 1 };
					\draw[gray,dashed] (\k,-4) -- (\k,-8.25);
					\draw[gray] (\k,1) -- (\k,-4);
				}
			\end{scope}
			
			\begin{scope}[shift={(0,-4)}, oriented WD, bbx = 1cm, bby =.4cm, bb min width=1cm, bb port sep=1.5]			
				\draw node [fill=\backgrnd,bb={1}{1}] (Tau) at (\bS -1,-1) {$t$};
				\draw node [fill=\backgrnd,bb={1}{2}, ] (Mu)  at ({2 * \bS - 1},-1) {$v$};
				\draw node [fill=\backgrnd,bb={1}{1}] (Nu)  at ({3 * \bS - 1},{2 * \kY}) {$u$};
				
				\draw (-1,-1) --     node[above] {$p_1$}       (0,-1)
				--                  node[above] {}          (Tau_in1);
				
				\draw (-1,-2) -- node[above] {$p_2$} (0,-2) -- (\bS-1, -2);
				\draw (-1,-3) -- node[above] {$p_3$} (0,-3) -- (\bS-1, -3);
				\draw (-1,-4) -- node[above] {$p_3$} (0,-4) -- (\bS-1, -4);

				\draw (Tau_out1) -- node[above] {$p_2$}    (Mu_in1);
				\draw (\bS-1,-2) -- (2*\bS-1, -2);
				\draw (\bS-1,-3) -- (2*\bS-1, -3);
				\draw (\bS-1,-4) -- (2*\bS-1, -4);
				
				\draw (Mu_out1) -- node[above] {$p_3$}    (3*\bS-1, -0.725);
				\draw (Mu_out2) -- node[above] {$p_4$}    (3*\bS-1, -1.325);
				
				\draw (2*\bS-1,-2) -- (3*\bS-1, -2);
				\draw (2*\bS-1,-3) -- (Nu_in1);
				\draw (2*\bS-1,-4) -- (3*\bS-1, -4);
				
				\draw (3*\bS-1,-0.725) to (4*\bS-2, -1.325) -- node[above] {$p_3$} (4*\bS-1, -1.325);
				\draw (3*\bS-1,-1.325) -- (3*\bS,-1.325) to (4*\bS-2, -4) -- node[above] {$p_4$} (4*\bS-1, -4);
				\draw (3*\bS-1,-2) to (4*\bS-2, -0.725) -- node[above] {$p_2$} (4*\bS-1, -0.725);
				\draw (Nu_out1) to (4*\bS-2, -3) -- node[above] {$p_4$} (4*\bS-1, -3);
				\draw (3*\bS-1,-4) to (4*\bS-2, -2) -- node[above] {$p_3$} (4*\bS-1, -2);
			\end{scope}
			
			\begin{pgfonlayer}{background}
				\filldraw [line width=4mm,join=round,\backgrnd]
				((-1,1.5)  rectangle (19,-8.5);
			\end{pgfonlayer}
			
		\end{tikzpicture}
	}
	\caption{An execution corresponding to a sequence of firings}\label{fig: execution alongside net}
\end{figure}
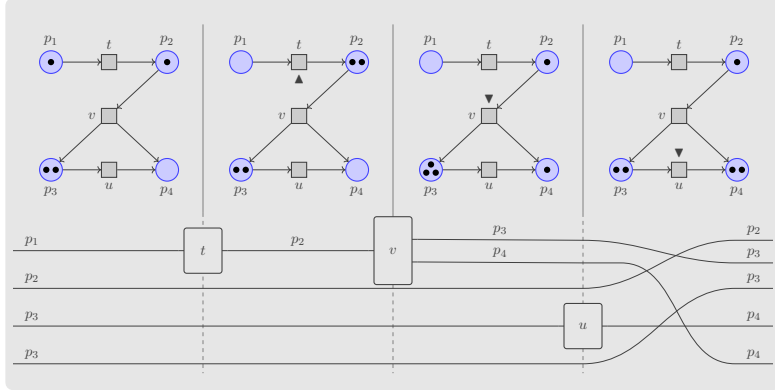

\noindent
This correspondence between Petri nets and free strict symmetric monoidal categories defines a \emph{process semantics} for nets: The FSSMC is interpreted as a deterministic version of its corresponding net, in which the history of every single token is tracked. This, in principle, could be implemented using a dependently typed language such as Idris~\cite{Brady2013} and, even more satisfactorily, such an implementation would allow one to map the FSSMC corresponding to a net to any other monoidal category representing a semantics, for instance the category $\textbf{Hask}$~\cite{HaskellWiki} of Haskell types and functions. The result is a procedure to formally compile a Petri net down to computations, with the net itself used as control to trigger the execution of algorithms and the passing of data between them. This is the approach adopted in defining the Statebox programming language~\cite{StateboxTeam2018}.

When focusing on implementation, this conceptual correspondence has to be made precise. In this paper we will review many approaches to the problem which have been pursued throughout the years, and highlight how all of them have to make compromises which are either computationally unfeasible or conceptually unsatisfying for our requirements. This is because the general strategy to link nets to FSSMCs has traditionally boiled down to the following:
\begin{itemize}
	\item We want a correspondence from nets to FSSMCs;
	\item We want said correspondence to be extendable to a functor;
	\item We want said functor to have an adjoint.
\end{itemize}
These desiderata are sensible from a purely categorical point of view, as best exemplified by the unifying work carried out in~\cite{Master2019}, but they are not enough to guarantee that a practical and usable implementation is feasible. So, instead, we propose the following "wishlist" to guide our implementation efforts:
\begin{definition}[List of requirements]\label{def: requirements}
	\leavevmode \\
	\begin{enumerate}
		\item We want to map each net to a free strict symmetric monoidal category, representing its possible executions. Free structures are particularly appealing for us since they are easier to handle computationally. Moreover we do not want to change the net/FSSMC definition too much to find a correspondence between the two, because we want to leverage on the diagrammatic formalism developed for both, see Requirement~\ref{item: mapping is useful for intuition}. \label{item: mapping net to FSSMC}
		\item We want the mapping from nets to FFSMCs to be computationally feasible. \label{item: mapping is computationally feasible}
		\item We want the FSSMCs corresponding to nets to represent computations in a \emph{meaningful way}. In particular, we want to be able to map net computations to other categories, translating the generating morphisms of a FSSMC into real operations (which can be pure functions, asynchronous calls etc.) via functors. Therefore, FSSMCs should not identify computations which are conceptually distinct. Moreover, any morphism in the FSSMC must correspond to a possible computation of the net. \label{item: mapping is faithful and full}
		\item We want our mapping to be useful for intuition: Users should be able to program in a goal-oriented way using nets (as explained in~\cite{Genovese}), test against their properties with already developed tools~(e.g.~\cite{Sobocinski2013a}), and then map the net to its computations to establish links with a semantics. \label{item: mapping is useful for intuition}
		\item We want users to be able to morph a net into another. Such morphisms should automatically be lifted to morphisms of net executions. We moreover want users to be able to ``tweak'' morphisms between executions directly in case the lifting provided is not satisfying for the task at hand. \label{item: morphisms are automatically lifted}
	\end{enumerate}
\end{definition}
In the following sections we will point out how the two lists provided above are somehow incompatible: If having a functorial correspondence from FSSMCs to nets doesn't pose any problem, going in the opposite direction begs for some choices to be made. These choices depend on the problem of linearizing multisets, and as we will see they cannot be globally extended to all nets without either quotienting computations which should be considered distinct, or by adding further structure to the definition of nets, which lessens their appeal in applications. So, while an adjoint correspondence between nets and computations seem desirable from a categorical point of view, all attempts to obtain one render it meaningless for applications.

We will also point out how improving upon the current situation is not possible if the functorial mapping from nets to FSSMCs is to be preserved. Luckily, dropping it is not really a problem: We will show how having a functor from FSSMCs to nets is enough for implementation purposes. In the last section, we introduce \idrisct~\cite{StateboxTeamb}, a work-in-progress implementation of the content we covered.
\section{Preliminary definitions}
When things are seen more in detail, the correspondence between Petri nets and FSSMCs is not as precise as it seems: It fails in a particularly frustrating way which, as pointed out in~\cite{Baldan2003}, boils down to the fact that the inputs and outputs of transitions are multisets and hence commutative monoids, while the monoid of objects in a monoidal category is not. To make these claims precise, we start by giving some definitions.
\subsection{Strings and multisets}
\begin{definition}
	Let $S$ be a set. We denote with $\Msets{S}$ and $\Strings{S}$ the set of \emph{finite multisets} and the set of \emph{strings of finite length} over $S$, respectively. The \emph{length of a string} $s$ is denoted with $|s|$, while the empty string with $\tensorUnit$.
\end{definition}
\begin{fact}
	Given a set $S$, $\Msets{S}$ is the \emph{free commutative monoid generated by $S$}, with the empty multiset as unit and multiset sum as multiplication. Similarly, $\Strings{S}$ is the \emph{free monoid generated by $S$}, with the empty string $\epsilon$ as unit and string concatenation as multiplication.
\end{fact}
\begin{fact}
	Any function between sets $f:A \to B$ gives rise to a corresponding multiset homomorphism $\Mset{f}: \Msets{A} \to \Msets{B}$. Similarly, any function between sets $f:A \to B$ gives rise to a corresponding monoid homomorphism $\Strings{f}: \Strings{A} \to \Strings{B}$.
\end{fact}
Note that the opposite is not true in general: There are multiset (monoid) homomorphisms that are not determined by a function between their corresponding base sets. This motivates the following definition:
\begin{definition}
	A multiset homomorphism $X: \Msets{A} \to \Msets{B}$ (respectively, monoid homomorphism $Y: \Strings{A} \to \Strings{B}$) is called \emph{grounded} if there is a function $f: A \to B$ such that $\Mset{f} = X$ (respectively $\Strings{f} = Y$).
\end{definition}
\begin{definition}
	Given a set $S$, we define \emph{the multiplicity of $S$} as the homomorphism $\Multiplicity{S}: \Strings{S} \to \Msets{S}$ sending a string $s$ to the multiset associating to each element of $S$ its number of occurrencies in $s$. When no ambiguity arises, we will use $\Multiplicity{}$ in place of $\Multiplicity{S}$.
\end{definition}
\begin{fact}
	Any monoid homomorphism $X: \Strings{A} \to \Strings{B}$ gives rise to a corresponding multiset homomorphism $\Mset{X}: \Msets{A} \to \Msets{B}$ by setting $\Mset{X}(a) = \Multiplicity{B}(X(a))$ for $a \in A$, and extending to non-generators using freeness.  If $X = \Strings{f}$ for some $f:A \to B$, then $\Mset{X} = \Mset{f}$.
\end{fact}
\subsection{Petri nets}
Now we focus on Petri nets. Many -- often inequivalent -- definitions of Petri nets have been given throughout the years, so it makes sense to spell out which particular definition we are committing to. In the context of process semantics, the following one is the most popular:
\begin{definition}
	A \emph{Petri net} $N$ is a tuple $\Net{N}$, where $\Pl{N}$ and $\Tr{N}$ are sets, called the set of \emph{places} and \emph{transitions} of $N$, respectively, while $\Pin{-}{N}$ and $\Pout{-}{N}$ are functions $\Tr{N} \to \Msets{\Pl{N}}$, representing the input/output places, respectively, connected to each transition.
\end{definition}
\begin{definition}
	A \emph{morphism of Petri nets} $f:N \to M$ is specified by a couple $(\PlM{f}, \TrM{f})$, with $\PlM{f}: \Msets{\Pl{N}} \to \Msets{\Pl{M}}$ multiset homomorphism and $\TrM{f}: \Tr{N} \to \Tr{M}$ function such that
	\begin{equation*}
		\Pin{-}{N};\PlM{f} = \TrM{f};\Pin{-}{M} \qquad \Pout{-}{N};\PlM{f} = \TrM{f};\Pout{-}{M}
	\end{equation*}
\end{definition}
Indeed, we get the following fact, which is proven noting that multiset homomorphisms can be composed and that identity multiset homomorphisms can be lifted to identity net morphsims.
\begin{fact}
	Petri nets and their morphisms form a category, denoted $\Petri$.
\end{fact}
The category $\Petri$ has been used, among others, in~\cite{Meseguer1990} and~\cite{Sassone1995}. Sometimes the definition of net morphism given above is too general, and begs for a suitable restriction. Most often, an interesting subcategory of $\Petri$ is the following one:
\begin{fact}
	There is a subcategory of $\Petri$, which we denote as $\GPetri$, where objects are nets and morphisms are \emph{grounded} homomorphisms of nets.
\end{fact}
The restriction from $\Petri$ to $\GPetri$ has been used in~\cite{Baez2018}, where nets were studied using double categories.
\subsection{Free strict symmetric monoidal categories (FSMCs)}
Now we spell out what a free symmetric monoidal category is. Here too there is some confusion, since freeness can be imposed only on objects, or on both objects and morphisms. Let us clarify:
\begin{definition}\label{def: free on objects smc}
	A \emph{free-on-objects, strict SMC (symmetric monoidal category)} is a strict symmetric monoidal category whose monoid of objects is freely generated.
\end{definition}
\begin{definition}\label{def: free smc}
	A \emph{free strict SMC (FSSMC)} is a symmetric monoidal category whose monoid of objects is $\Strings{S}$ for some set of generators $S$, and whose morphisms are generated by the following introduction rules:
	\begin{gather*}
		\frac{s \in \Strings{S}}{\id{s}: s \to s}
		\qquad 
		\frac{s,t \in \Strings{S}}{\sigma_{s,t}: {s \tensor t} \to {s \tensor t}} 
		\qquad
		\frac{(\alpha, s,t) \in T}{\alpha: s \to t} \\\\
		\frac{\alpha: A \to B, \,\, \alpha': A' \to B'}{\alpha \tensor \alpha':  {A \tensor A'} \to {B \tensor B'}} 
		\qquad 		
		\frac{\alpha: A \to B, \,\, \beta: B \to C}{\alpha;\beta:  A \to C}
	\end{gather*}
	Where elements of the set $T$ are triples $(\alpha, s, t)$ with $s, t \in \Strings{S}$. Morphisms are quotiented by the following equations, for $\alpha: A \to B$, $\alpha': A' \to B'$, $\alpha'': A'' \to B''$, $\beta: B \to C$, $\beta':B' \to C'$, $\gamma: C \to D$:
	\begin{align*}
		\alpha ; \id{B} = &\,\,\alpha = \id{A} ; \alpha & 
		\quad  (\alpha;\beta);\gamma &= \alpha;(\beta;\gamma)\\
		\tensorUnit\tensor \alpha = &\,\,\alpha = \alpha \tensor \tensorUnit& 
		\quad  (\alpha \tensor \alpha') \tensor \alpha'' &= \alpha \tensor (\alpha' \tensor \alpha'')\\
		\id{A} \tensor \id{A'} &= \id{A \tensor A'} & 
		\quad  (\alpha \tensor \alpha') ; (\beta \tensor \beta') &= (\alpha ; \beta) \tensor (\alpha' ; \beta')\\
		\sigma_{A, A' \tensor A''} = (\sigma_{A,A'} &\tensor \id{A''}); (\id{A} \tensor \sigma_{A',A''}) & 
		\quad  \sigma_{A,A'};\sigma_{A',A} &= \id{A \tensor A'}\\
 		&\quad\qquad \mathrlap{\sigma_{A,A'};(\alpha' \tensor \alpha)= (\alpha \tensor \alpha');\sigma_{B,B'}}&
	\end{align*}
\end{definition}
Intuitively, a FSSMC is a symmetric monoidal category where morphisms do not satisfy any further equation. Since the monoid of objects is freely generated, we will often use the string notation $s_1 \dots s_n$ to denote $s_1\tensor \dots \tensor s_n$. Among FSSMCs there is a distinguished class of categories -- which will be the cause of all the negative results in this paper -- that deserve their own notation:
\begin{definition}
	Given a set $S$ we denote with $\Sym{S}$ the category generated as in Definition~\ref{def: free smc}, where the monoid of objects is $\Strings{S}$ and the set of generating morphisms $T$ is empty.
\end{definition}
Finally, we have the following couple of results, easy to prove:
\begin{fact}
	Free-on-objects strict SMCs and strict symmetric functors between them form a category, denoted $\FOSSMC$. Restricting the objects to FSSMCs, we obtain a subcategory $\FSSMC \subset \FOSSMC$. Restricting strict symmetric monoidal functors to be grounded homomorphisms on objects, we obtain further restrictions $\GFOSSMC \subset \FOSSMC$ and $\GFSSMC \subset \FSSMC$.
\end{fact}
\section{Past approaches and development implications}
We now review past approaches to the problem of finding a correspondence between nets and FSSMCs, focusing on how they relate to our requirements.
\subsection{The symmetric approach}
In~\cite{Meseguer1990}, probably the most influential paper describing Petri nets from a categorical standpoint, nets and their morphisms are axiomatized as the objects and morphisms of a category. In~\cite{Degano1996} this process is taken a step further, with each Petri net being mapped to its category of computations, which is strict symmetric monoidal. Since inputs and outputs of transitions are multisets, a linearization problem arises, because objects in a symmetric monoidal category do not commute in the general case. The authors solve the problem by imposing commutativity: In~\cite{Sassone1996} it has been proven that the mapping in~\cite{Degano1996} is equivalent to mapping each net $N$ to a category as in Definition~\ref{def: free smc}, with $T$ being the set of transitions of the net, and then quotienting by the following equations, for $s, s',t$ object generators, $s \neq t$ and $(\alpha, A \otimes s \otimes t \otimes B, A' \otimes s' \otimes s' \otimes B')$ in $T$: 
\begin{equation*}
	\sigma_{s,t} = \id{s,t}
	\qquad
	(\id{A} \otimes \sigma_{s,s} \otimes \id{B});\alpha = \alpha
	\qquad 
	\alpha;(\id{A'} \otimes \sigma_{s',s'} \otimes \id{B'}) = \alpha
\end{equation*}
The idea is to completely annihilate symmetries, making morphisms totally indifferent to the causal relationships between tokens. 

This approach is unsatisfying for many reasons: First of all, as proven in~\cite{Sassone1995}, annihilating symmetries implies that the correspondence from nets to categories cannot be made functorial in any straightforward way. More importantly the category of computations of a net, defined in this way, \emph{has no real computational meaning}: It is not possible to keep track of the causal flow of tokens, and given that symmetric monoidal categories are almost never commutative on objects, the idea of mapping computations to a semantics is shattered. In particular, commutativity on objects is not satisfied by functional programming languages when we consider them as categories with data types as objects and functions between them as morphisms: Here, the monoidal product amounts to taking tuples, which do not commute. Such a strategy then violates Requirement~\ref{item: mapping is faithful and full} in Definition~\ref{def: requirements}.

On the other hand, in~\cite{Baez2018} the authors worked with grounded morphisms -- hence in~$\GPetri$ -- and defined a double-categorical framework which allows the gluing together of nets. The idea is very interesting, since it can be employed in using Petri nets to write code modularly. Moreover, if the choice of working in~$\GPetri$ reduces the expressiveness of Petri net morphisms, it doesn't necessarily constitute a severe limitation from the applicative standpoint. The additional double-categorical structure allows us to obtain a mapping between nets and symmetric monoidal categories which is, this time, functorial in a double-categorical interpretation. Unfortunately, to make things work, the authors had to impose commutativity on places once again, violating Requirement~\ref{item: mapping is faithful and full}.
\subsection{The pre-net approach}
In~\cite{Baldan2003}, commutativity was dropped in favor of another strategy, namely weakening the definition of Petri net to the one of \emph{pre-net}, and showing how pre-nets present free strict symmetric monoidal categories. The advantage is that pre-nets can be functorially -- and adjointly -- mapped to FSSMCs. Pre-nets are essentially an ordered version of Petri nets where the order in which places enter/exit transitions has to be specified explicitly.

This approach works, but the requirement that the ordering play nice with net morphisms gets in the way of using Petri nets to model complicated processes. Specifically, there may not be any morphisms between two pre-nets, even if there are between their underlying nets. This is shown in Figure~\ref{fig: pre-net failing}, where we have adopted the graphical formalism for pre-nets introduced in~\cite{Baldan2003}.
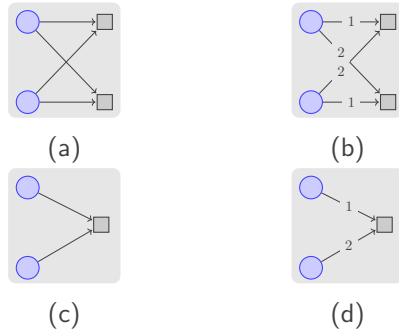
\begin{figure}[h]
	\centering
	\begin{subfigure}[t]{0.24\textwidth}\centering
		\scalebox{0.5}{
		\begin{tikzpicture}
			\node [place,tokens=0] (1a)  {};
	
			\node [place,tokens=0] (1b)  [below = 1.5cm of 1a] {};
			
			\node [transition] (2a) [right = 1.5cm of 1a]     {}
				edge [pre] node[] {} (1a)
				edge [pre] node[] {} (1b);
			\node [transition] (2b) [right = 1.5cm of 1b]     {}
				edge [pre] node[] {} (1a)
				edge [pre] node[] {} (1b);
		
			\begin{pgfonlayer}{background}
				\filldraw [line width=4mm,join=round,black!10]
				(1a.north -| 1a.west) rectangle (2b.south -| 2b.east);
			\end{pgfonlayer}
		\end{tikzpicture}
		}
	\caption{}\label{fig: first net}
	\end{subfigure}
	\begin{subfigure}[t]{0.24\textwidth}\centering
		\scalebox{0.5}{
		\begin{tikzpicture}
			\node [place,tokens=0] (1a)  {};
			
			\node [place,tokens=0] (1b)  [below = 1.5cm of 1a] {};
			
			\node [transition] (2a) [right = 1.5cm of 1a]     {}
				edge [pre] node[fill=black!10] {1} (1a)
				edge [pre] node[fill=black!10, below left] {2} (1b);
			\node [transition] (2b) [right = 1.5cm of 1b]     {}
				edge [pre] node[fill=black!10, above left] {2} (1a)
				edge [pre] node[fill=black!10] {1} (1b);
			
			\begin{pgfonlayer}{background}
				\filldraw [line width=4mm,join=round,black!10]
				(1a.north -| 1a.west) rectangle (2b.south -| 2b.east);
			\end{pgfonlayer}
		\end{tikzpicture}
		}
	\caption{}\label{fig: first pre-net}
	\end{subfigure}
	\\
	\begin{subfigure}[t]{0.24\textwidth}\centering
		\scalebox{0.5}{
		\begin{tikzpicture}
			\node [place,tokens=0] (1a)  {};
			
			\node [place,tokens=0] (1b)  [below = 1.5cm of 1a] {};
			
			\node [transition] (2a) [below right = 0.55cm and 1.5cm of 1a]     {}
				edge [pre] node[] {} (1a)
				edge [pre] node[] {} (1b);
			
			\begin{pgfonlayer}{background}
				\filldraw [line width=4mm,join=round,black!10]
				(1a.north -| 1a.west) rectangle (2b.south -| 2b.east);
			\end{pgfonlayer}	
		\end{tikzpicture}
		}
	\caption{}\label{fig: second net}
	\end{subfigure}
	\begin{subfigure}[t]{0.24\textwidth}\centering
		\scalebox{0.5}{
		\begin{tikzpicture}
			\node [place,tokens=0] (1a)  {};
			
			\node [place,tokens=0] (1b)  [below = 1.5cm of 1a] {};
			
			\node [transition] (2a) [below right = 0.55cm and 1.5cm of 1a]     {}
				edge [pre] node[fill=black!10] {1} (1a)
				edge [pre] node[fill=black!10] {2} (1b);
			
			\begin{pgfonlayer}{background}
				\filldraw [line width=4mm,join=round,black!10]
				(1a.north -| 1a.west) rectangle (2b.south -| 2b.east);
			\end{pgfonlayer}
		\end{tikzpicture}
		}
	\caption{}\label{fig: second pre-net}
	\end{subfigure}
	\caption{The net~\ref{fig: first net} is morphed to the net~\ref{fig: second net} by sending places to themselves and the two transitions in~\ref{fig: first net} to the one in~\ref{fig: second net}. There is no way to lift this to a morphism between pre-nets~\ref{fig: first pre-net} and~\ref{fig: second pre-net} without collapsing places.}\label{fig: pre-net failing}
\end{figure}

\noindent
Indeed, there is no functor from the category of Petri nets to the category of pre-nets, meaning that the ``specification of a net'' -- which is how pre-nets are interpreted in~\cite{Baldan2003} -- cannot be created on the fly in a way that respects morphisms of underlying nets. To understand why this is a problem, imagine the following scenario: A user draws a Petri net. To execute it, it is automatically mapped to its corresponding FSSMC. When the user runs the net, selecting which tokens have to be processed by which transition via the UI, the corresponding morphisms are composed in the FSSMC. Now the user decides to morph the net into another. It would be desirable to induce a functor between the corresponding FSSMCs, so that we could ``import'' all the histories from the former to the latter. Using pre-nets as a stepping stone between nets and FSSMCs we cannot, because the net transformation specified by the user may not correspond to a morphism of pre-nets. This violates Requirement~\ref{item: morphisms are automatically lifted} in Definition~\ref{def: requirements}, the only solution being to ask users to employ pre-nets directly. In doing so the graphical formalism becomes less intuitive and the allowed transformations greatly restricted. Even worse, it is not easy to intuitively understand when a net morphism is ``allowed'' and when it is not, with an obvious -- and substantial -- loss in applications. This violates Requirement~\ref{item: mapping is useful for intuition}. 
\subsection{The quotient approach}
In~\cite{Sassone1995}, Petri nets are taken as they are, and it's free strict symmetric monoidal categories being modified to accomodate an adjunction. We consider this approach as one of the most valid so far, and we generalized it in~\cite{Genovese2018}. Let us spell things out in detail:
\begin{definition}[From~\cite{Sassone1995}, Def. 3.5]\label{def: unordered executions}
	Given a Petri net $(N)$, we map it to the free-on-objects, strict symmetric monoidal category $\FoldSym N$ defined as follows:
	\begin{itemize}
		\item The monoid of objects of $\FoldSym N$ is freely generated by $\Pl{N}$;
		\item Morphisms are generated by the following introduction rules:
		\begin{gather*}
			\frac{s \in \Strings{\Pl{N}}}{\id{s}: s \to s}
			\quad \frac{s,t \in \Strings{\Pl{N}}}{\sigma_{s,t}: {s \tensor t} \to {t \tensor s}}\\\\
			\frac{t \in T_N}{t_{u,v}: u \to v} \quad \forall u,v.(\Multiplicity{}(u)= \Pin{t}{N} \wedge \Multiplicity{} (v)= \Pout{t}{N})\\\\
			\frac{\alpha: A \to B, \,\, \alpha': A' \to B'}{\alpha \tensor \alpha':  {A \tensor A'} \to {B \tensor B'}} 
			\qquad 		
			\frac{\alpha: A \to C, \,\, \beta: B \to C}{\alpha;\beta:  A \to C}
		\end{gather*}
		\item Morphisms are quotiented by the same axioms of Definition~\ref{def: free smc} plus the following one, for $p: u \to u'$, $q: v \to v'$ symmetries:
		\begin{equation}\label{eq: Sassone axiom}
			p;t_{u',v'} = t_{u,v};q
		\end{equation}
	\end{itemize}
\end{definition}
This strategy works, but is computationally unfeasible: For each transition $t$, we need to introduce as many generating morphisms as are the permutations of the input/output places of $t$. If for instance $t$ has $10$ different inputs and no outputs, we need to introduce $10!$ generating morphisms. Moreover, they have to be linked together by quotienting the category with supplementary equations, resulting in more computational overhead. (Quotients are difficult to deal with in code.) Although the quotienting overhead can be reduced (for instance using \texttt{postulate} in Idris), the superexponential explosion of generating morphisms is not avoidable. All this violates Requirement~\ref{item: mapping is computationally feasible}.

Furthermore, axiom~\ref{eq: Sassone axiom} in Definition~\ref{def: unordered executions} makes $\FoldSym(N)$ non-free, violating Requirement~\ref{item: mapping net to FSSMC}. This depends on the fact that axiom~\ref{eq: Sassone axiom} doesn't just relate generating morphisms to other generating morphisms, but also to themselves, in case they have repeated entries in their source and targets. Using wiring diagrams to represent morphisms, for a transition $t: s \otimes s \to v$ axiom~\ref{eq: Sassone axiom} entails that the diagrams in Figure~\ref{fig: sassone fails} are considered equal.
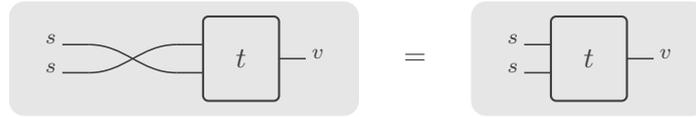
\begin{figure}[!h]
	\centering
	\begin{tikzpicture}[oriented WD, bb port length=10pt]
	\node[bb={2}{1}] (X1) {$t$};
	\draw[label] 
	node [yshift=2, right=2pt of X1_out1] {$v$}
	;

	\node[bb={2}{0}, transparent, left = 0.5cm of X1] (T1) {$t$};
	\draw[label] 
	node [yshift=2, left=2pt of T1_in1] {$s$}
	node [yshift=2, left=2pt of T1_in2] {$s$}
	;
	
	\draw (T1_in1') to (X1_in2);
	\draw (T1_in2') to (X1_in1);
	
	\node[right = 1.5cm of X1] (eq) {$=$};
	
	\node[bb={2}{1}, right = 1.5cm of eq] (X2) {$t$};
	\draw[label] 
	node [yshift=2, left=2pt of X2_in1] {$s$}
	node [yshift=2, left=2pt of X2_in2] {$s$}
	node [yshift=2, right=2pt of X2_out1] {$v$}
	;
	
	\begin{pgfonlayer}{background}
	\filldraw [line width=4mm,join=round,\backgrnd]
	([xshift=-5mm]X1.north -| T1_in1.west) rectangle ([xshift=5mm]X1.south -| X1_out1.east);
	
	\filldraw [line width=4mm,join=round,\backgrnd]
	([xshift=-5mm]X2.north -| X2_in1.west) rectangle ([xshift=5mm]X2.south -| X2_out1.east);
	\end{pgfonlayer}
	
	\end{tikzpicture}
	\caption{According to Definition~\ref{def: unordered executions} these diagrams are equal, making the category non-free}\label{fig: sassone fails}
\end{figure}

\noindent
This is a problem when it comes to mapping net computations to a semantics: As tuples are not commutative in the general case, our functorial correspondence cannot be symmetric monoidal.

Finally, the notion of morphism between the categories $\FoldSym(N)$ is greatly impractical, since morphisms between them are equivalence classes of functors. This equivalence collapses functors with radically different behaviors if we embrace the idea that symmetries are important to keep track of ``which-token-is-doing-what'' in a net. This is in contrast with Requirement~\ref{item: mapping is faithful and full}.
\section{The curse of linearization}
As we have seen, many approaches come very close to showing adjunctions or even equivalences between something that very closely resembles the category $\Petri$ and something that very closely resembles the category $\FSSMC$. However, no approach really nails down a ``computer-friendly'' correspondence between $\Petri$ and $\FSSMC$ themselves. Why is that? As we will see, the issue is that there is no canonical way to linearize multiset homomorphisms that is preserved by symmetric monoidal functors.

At the most basic level, to send transitions of a net to generators in a FSSMC we need to linearize its inputs and outputs from multisets to strings, hence we need to have a function $\Ordering{S}: \Msets{S} \to \Strings{S}$ such that $\Ordering{S};\Multiplicity{} = \id{\Msets{S}}$ for every set of generators $S$.
For obvious reasons, there is no canonical choice for such a function without imposing additional structure on the generating set and, according to our desiderata, we don't want this additional structure to ``get in the way'' when we use nets for practical purposes.

Luckily, there are many different ways to obtain $\Ordering{S}$ without having to impose unreasonable requirements on the net. We present one of the many possible approaches in the Appendix~\ref{app: Appendix}, and for now we just suppose to have $\Ordering{}$ defined. With it, we can cleanly linearize transition inputs/outputs and generate a FSSMC from a  net.
\begin{definition}\label{def: free functor on objects}
	Given a Petri net $N$, we map it to the FSSMC $\OFoldSym N$ generated as in Definition~\ref{def: free smc}, with $\Pl{N}$ as the set of object generators, and generating morphisms given by
	\begin{equation*}
		T := \suchthat{(t,\Ordering{N}(\Pin{t}{N}),\Ordering{N}(\Pout{t}{N}))}{t \in \Tr{N}}
	\end{equation*}
\end{definition}
It is useful to compare our mapping with the one provided in Definition~\ref{def: unordered executions}: The main difference between the two is that the use of the ordering function $\Ordering{N}$ in Definition~\ref{def: free functor on objects} is replaced by a universally quantified statement over multiplicities using $\Multiplicity{}$ in Definition~\ref{def: unordered executions}. The computational overhead introduced by Definition~\ref{def: unordered executions} ultimately depends on the fact that the function $\Ordering{N}$, which provides a canonical choice for generating morphisms, cannot be defined without imposing further structure on the net. On the contrary, in Definition~\ref{def: free functor on objects} we are able to use it to reduce the number of generating morphisms introduced from $n! \cdot m!$ in the worst case to just $1$, while getting rid of quotients altogether and keeping the category free at the same time.

We conclude the section with the observation that the mapping from nets to FSSMCs is invertible, the proof of which is obvious from the definitions:
\begin{definition}\label{def: forgetful functor on objects}
	To every FSSMC $\CategoryC$, we associate the net $\UnFoldSym \CategoryC$ defined as follows:
	\begin{itemize}
		\item Places of $\UnFoldSym \CategoryC$ are the generating objects of $\CategoryC$;
		\item Transitions of $\UnFoldSym \CategoryC$ are the generating morphisms of $\CategoryC$;
		\item $\Pin{t}{\UnFoldSym \CategoryC}$, for $t: A \to B$ generating morphism of $\CategoryC$, is defined as $\Multiplicity{}(A)$;
		\item $\Pout{t}{\UnFoldSym \CategoryC}$, for $t: A \to B$ generating morphism of $\CategoryC$, is defined as $\Multiplicity{}(B)$.
	\end{itemize}
\end{definition}
\begin{proposition}\label{prop: fold-unfold isomorphism}
	For any net $N$, $\UnFoldSym \OFoldSym N$ is isomorphic to $N$. For each FSSMC $\CategoryC$, $\OFoldSym \UnFoldSym \CategoryC$ is isomorphic to $\CategoryC$.
\end{proposition}
\subsection{Transition-preserving functors}
We are quite happy: There are ways to linearize the input/output places of Petri nets that allow to biject them to FSSMCs. We now need to find a suitable notion of morphism between FSSMCs which can be considered as the linearized counterpart of a net morphism. In our conceptual framework, morphism generators represent net transitions, while symmetries and identities represent the ``necessary bookkeeping'' to deterministically spell out the causal flow of tokens. So, it seems reasonable to restrict to functors that send generating morphisms to generating morphisms. We make this precise:
\begin{definition}
	Let $\CategoryC, \CategoryD$ be FSSMCs. A strict monoidal functor $F:\CategoryC \to \CategoryD$ is called \emph{transition-preserving} if it maps generating morphisms in $\CategoryC$ to morphisms in $\CategoryD$ of the form $\sigma;t;\sigma'$, where $t$ is a generating morphism and $\sigma, \sigma'$ are symmetries.
\end{definition}
Notice how functors between FSSMCs as defined in~\cite{Baldan2003} are transition-preserving, as are representatives of the equivalence classes that are used in~\cite{Sassone1995} to define functors. It can easily be proven that our definition behaves well with respect to the categorical structure of $\FSSMC$, and in fact:
\begin{proposition}\label{prop: transition-preserving subcategory}
	FSSMCs and transition-preserving functors form a subcategory of $\FSSMC$. Similarly, FSSMCs and grounded, transition-preserving functors form a subcategory of $\GFSSMC$. From now on we will use $\FSSMC$ and $\GFSSMC$ to denote such restricted categories.
\end{proposition}
\begin{proof}
	Clearly, identity functors are transition-preserving. Now let $F:\CategoryC \to \CategoryD$ and $G: \CategoryD \to \CategoryE$. If $t_\CategoryC$ is a generator of $\CategoryC$, then it is mapped by $F$ to $\sigma_F;t_\CategoryD;\sigma'_F$ in $\CategoryD$. $t_\CategoryD$ is itself a generator, so it is mapped by $G$ to $\sigma_G; t_\CategoryE; \sigma'_G$. Putting all of this together, we find that $F;G$ maps the generator $t_\CategoryC$ to:
	\begin{align*}
		GFt_\CategoryC &= G(\sigma_F;t_\CategoryD;\sigma'_F)\\
		&= G\sigma_F ; Gt_\CategoryD; G\sigma'_F\\
		&= G\sigma_F; (\sigma_G;t_\CategoryE;\sigma'_G); G\sigma_F\\
		&= (G\sigma_F; \sigma_G);t_\CategoryE;(\sigma'_G; G\sigma_F)
	\end{align*}
	Because $G$ is strict monoidal, $G\sigma_F$ and $G\sigma'_F$ are symmetries. Hence, $G\sigma_F;\sigma_G$ and $G\sigma'_F;\sigma'_G$ are symmetries as well, proving that $F;G$ is transition-preserving.
\end{proof}
The choice of restricting to transition-preserving functors does not break any of the requirements spelled out in Definition~\ref{def: requirements}: Non-transition-preserving functors do not have any interpretation as morphisms of net computations.

It is also worth noting that the isomorphism between $\CategoryC$ and $\OFoldSym \UnFoldSym \CategoryC$ of Proposition~\ref{prop: fold-unfold isomorphism} is transition-preserving, so we can still ``go back and forth'' from nets to computations, and vice-versa, in our restricted setting.

To be sure that our definition is sensible, we still have to check that we can lift net morphisms to transition-preserving functors. What we have in mind is something along the lines of the following proposition:
\begin{proposition}[Sketch]\label{def: free functor on morphisms}
	Let $f: N \to M$ be a morphism of nets. $f$ induces a transition-preserving functor $\OFoldSym f: \OFoldSym N \to \OFoldSym M$ as follows:
	\begin{itemize}
		\item A generating object $s$ of $\OFoldSym N$ is a place in $N$, so we set $(\OFoldSym f)s = \Ordering{M} \PlM{f}(s)$;
		\item We extend the mapping $\OFoldSym f$ to all objects by using the fact that the monoid of objects of $\OFoldSym N$ is free;
		\item On morphisms, we send identities to identities and symmetries to symmetries. If $t: \Ordering{N}\Pin{t}{N} \to \Ordering{N}{\Pout{t}{N}}$ is a generator of $\OFoldSym N$, we send it to:
		\begin{equation*}
			(\OFoldSym f) \Ordering{N}\Pin{t}{N} \xrightarrow{\bar{\sigma}} \Ordering{M}\Pin{\TrM{f}t}{M} \xrightarrow{\TrM{f}t} \Ordering{M}\Pout{\TrM{f}t}{M} \xrightarrow{\bar{\sigma}'} (\OFoldSym f) \Ordering{N}\Pin{t}{N}
		\end{equation*}
		\item We extend $\OFoldSym f$ to all the remaining morphisms by (monoidal) composition.
	\end{itemize}
	If $f$ is grounded, so is $\OFoldSym f$.
\end{proposition}
The way we set things up in Proposition~\ref{def: free functor on morphisms} seems to be the only sensible thing to do. Nevertheless, there is a problem: For each generating morphism of $\OFoldSym N$, how do we define the symmetries $\bar{\sigma}$, $\bar{\sigma}'$? This is the key point that makes getting a functorial correspondence from nets to FSSMCs so difficult: Linearizing multisets to strings is easy, while linearizing multiset \emph{homomorphisms} to string homomorphisms is not! Before making this precise, we prove that there are choices of $\bar{\sigma}, \bar{\sigma}'$ for which Proposition~\ref{def: free functor on morphisms} holds.
\begin{proof}[Proof of Proposition~\ref{def: free functor on morphisms}]
	On objects there is nothing to prove, since the mapping is obtained by applying freeness on the mapping on generators, which makes it monoidal by definition. This is also sufficient to prove the last statement: If $f$ is a grounded homomorphism of multisets on places, $\OFoldSym f$ is a grounded monoid homomorphism on objects.
	
	On morphisms, identities are mapped to identities and symmetries to symmetries, so symmetry and identity preservation laws hold trivially. We need to check that the functor is well-defined on generating morphisms, hence proving:
	\begin{equation*}
		\Multiplicity{}((\OFoldSym f) \Ordering{N}\Pin{t}{N}) =  		
		\Multiplicity{}(\Ordering{M}\Pin{\TrM{f}t}{M}) 
		\quad
		\Multiplicity{}(\Ordering{M}\Pout{\TrM{f}t}{M}) = 
		\Multiplicity{}((\OFoldSym f) \Ordering{N}\Pin{t}{N})
	\end{equation*}
	We focus on the first one, the proof of the second being analogous. $\Ordering{N}\Pin{t}{N}$ is an object of $\OFoldSym N$, and so it is a string $s_1 \dots s_n$. By definition, 
	\begin{align*}
		(\OFoldSym f) \Ordering{N}\Pin{t}{N} 
		&= (\OFoldSym f) (s_1 \dots s_n) \\
		&= (\OFoldSym f) s_1 \dots (\OFoldSym f) s_n \\
		&= \Ordering{M} \PlM{f}(s_1) \dots \Ordering{M} \PlM{f}(s_n)
	\end{align*}
	Hence, taking multiplicities, 
	\begin{align*}
		\Multiplicity{}((\OFoldSym f) \Ordering{N}\Pin{t}{N}) 
		&= \Multiplicity{}(\Ordering{M} \PlM{f}(s_1) \dots \Ordering{M} \PlM{f}(s_n)) \\
		&= \Multiplicity{}(\PlM{f}(s_1) \dots \PlM{f}(s_n)) \\
		&= \PlM{f}\Multiplicity{}(s_1 \dots s_n) \\
		&= \PlM{f}(\Pin{t}{N}) \\
		&= \Pin{\TrM{f}t}{M} \\
		&= \Multiplicity{}(\Ordering{M}\Pin{\TrM{f}t}{M}) \\
		&= \Multiplicity{}((\OFoldSym f) \Ordering{N}\Pin{t}{N})
	\end{align*}
	This is enough to guarantee that $(\OFoldSym f) \Ordering{N}\Pin{t}{N}$ and $\Ordering{M}\Pin{\TrM{f}t}{M}$ are permutation of one another, so there exists a symmetry $\bar{\sigma}$ between them. An obvious consequence is also that $\OFoldSym f$ is transition-preserving, by definition.
	
	Preservation of monoidal products and compositions holds trivially since we defined them freely from generating morphisms, identities and symmetries. 
\end{proof}
Proposition~\ref{def: free functor on morphisms} says that if for each generating morphism we can pick a symmetry -- any symmetry -- between $(\OFoldSym f) \Ordering{N}\Pin{t}{N}$ and $\Ordering{M}\Pin{\TrM{f}t}{M}$ on one hand, and between $\Ordering{M}\Pout{\TrM{f}t}{M}$ and $(\OFoldSym f) \Ordering{N}\Pin{t}{N}$ on the other, then we get a functor. Clearly we want this choice to respect functor composition. Alas, this is not possible without imposing further structure on the nets in a way that violates our requirements. We will demonstrate this by investigating the structure of symmetries in a FSSMC.
\subsection{Chasing symmetries}
\begin{proposition}\label{prop: different objects unique symmetry}
	Let $\sigma: s \to s'$ be a morphism in $\Sym{S}$. If $s$ can be written as a monoidal product $s_1 \dots s_n$ of elements of $S$, with $s_i = s_j$ iff $i = j$, then for any other symmetry $\sigma': s \to s'$ it is $\sigma = \sigma'$.
\end{proposition}
\begin{proof}
	Using the coherence theorem for string diagrams of symmetric monoidal categories~\cite{Selinger2010}, $\Sym{S}$ is a category where objects are wires and morphisms are identities and crossings. So, $\sigma:s \to s'$ can be represented as a bunch of wires, each going from an $s_i$ in $s$ to exactly one $s'_j$ in $s'$, with $s_i = s'_j$. Coherence for symmetric monoidal categories means that only connectivity matters, that is, two symmetries $\sigma, \sigma'$ are considered equal if the wiring of $\sigma$ can be topologically rearranged into the wiring of $\sigma'$, keeping the ends of the wires fixed in their positions, and without bending wires into a ``U'' shape.
	
	Since $\Strings{S}$ is free and generated by $S$, the decomposition of $s$ into object generators $s_1 \dots s_n$ is unique. The same can be said of $s'$ with $s' = s'_1 \dots s'_m$, and $\sigma: s \to s'$ implies that $n = m$. Given that all $s_i$ are distinct, each $i$ has exactly one corresponding $j$ such that $s_i = s'_j$, and $\sigma$ must connect the two. This completes the proof, since any other $\sigma$ is forced to make the same connection between $s_i$ and $s'_j$ for an arbitrary $i$, and by coherence, connectivity of ends is the only thing that matters to rearrange a tangle of wires into another. 
\end{proof}
Notice how the proposition above does not hold if we allow repeated generating objects in the decomposition of $s$. In fact, if $s = s_1 s_1$ for some generating object $s_1$, the morphisms below cannot be topologically deformed into one another:
\begin{figure}[!h]
	\centering
		\begin{tikzpicture}[oriented WD, bb port length=10pt]
		\node[bb={2}{2}, transparent] (X1) {};
		\draw[label] 
		node [yshift=2, left=2pt of X1_in1] {$s_1$}
		node [yshift=2, left=2pt of X1_in2] {$s_1$}
		node [yshift=2, right=2pt of X1_out1] {$s_1$}
		node [yshift=2, right=2pt of X1_out2] {$s_1$}
		;
		
		\draw (X1_in1') to (X1_out2');
		\draw (X1_in2') to (X1_out1');
		
		\node[bb={2}{2}, transparent, right = 5cm of X1] (X2) {};
		\draw[label] 
		node [yshift=2, left=2pt of X2_in1] {$s_1$}
		node [yshift=2, left=2pt of X2_in2] {$s_1$}
		node [yshift=2, right=2pt of X2_out1] {$s_1$}
		node [yshift=2, right=2pt of X2_out2] {$s_1$}
		;

		\draw (X2_in1') to (X2_out1');
		\draw (X2_in2') to (X2_out2');
		
		\begin{pgfonlayer}{background}
			\filldraw [line width=4mm,join=round,\backgrnd]
			([xshift=-5mm]X1.north -| X1_in1.west) rectangle ([xshift=5mm]X1.south -| X1_out1.east);
		
			\filldraw [line width=4mm,join=round,\backgrnd]
			([xshift=-5mm]X2.north -| X2_in1.west) rectangle ([xshift=5mm]X2.south -| X2_out1.east);
		\end{pgfonlayer}
		
		\end{tikzpicture}
\end{figure}

\noindent
Focusing on the counterexamples a bit more, we see that the problem with repeated generating objects is that they may or may not be swapped, and that these two choices are not equivalent. Luckily, there is a clean way to establish whether or not a symmetry swaps the same object generators.
\begin{definition}
	A morphism in $\Sym{S}$ is called a \emph{basic block} if it is of the form $\id{u} \tensor \sigma_{s_1,s_2} \tensor \id{t}$ for some objects $u,t$ and some object generators $s_1,s_2$.
\end{definition}
\begin{fact}
	Using coherence conditions, any symmetry $\sigma: s \to s'$ can be written as a finite composition of basic blocks.
\end{fact}
\begin{definition}\label{def: lifting strings}
	Let $S$ be a set. Given a string $s \in \Strings{S}$, denote with $s^l$ the string obtained as follows:
	\begin{itemize}
		\item Starting from the left, consider the $i$-th entry of $s$, call it $s_i$. It will be equal to some generator $t$. 
		\item If we have $s_i \neq s_j$ for all $j < i$, then set $s_i$ equal to the formal symbol $t^0$;
		\item Otherwise, there are other entries $j$ with $j < i$ and both $s_i = t = s_j$. Consider the largest of such $j$. If $t$ has been replaced with the formal symbol $t^k$ at entry $j$, replace it with $t^{k+1}$ at entry $i$.
	\end{itemize}
\end{definition}
\begin{example}
	Consider the string $s = aababbccba$ on the set $\{a,b,c\}$. Applying Definition ~\ref{def: lifting strings}, we get $s^l =  a^1 a^2 b^1 a^3 b^2 b^3 c^1 c^2 b^4 a^4$. 
\end{example}
$s^l$ is just a rewriting of $s$ in which all equal entries have been distinguished by enumeration. It follows trivially that $s^l$ is a string itself:
\begin{proposition}
	If $s \in \Strings{S}$, then $s^l \in \Strings{{S^l}}$, where
	\begin{equation*}
		S^l := \suchthat{s^i}{s \in S \wedge i \in \naturals}
	\end{equation*}
	Moreover, there is an obvious functor $\Sym{S^l} \to \Sym{S}$ which collapses $s^i$ to $s$ for all $i \in \naturals$.
\end{proposition}
We now extend the procedure we defined for strings to symmetries between them, as follows:
\begin{definition}\label{def: lifting symmetries}
	Let $\sigma_1;\dots;\sigma_n: s \to s'$ be a composition of basic blocks in $\Sym{S}$. Denote by $\sigma_1^l;\dots;\sigma_n^l$ the symmetry obtained by using the following procedure:
	\begin{itemize}
		\item Suppose $\sigma_1 := \id{u} \tensor \sigma_{s_1,s_2} \tensor \id{t}$. Define $\sigma_1^l$ by replacing $s_1, s_2$ and every generating object in $u,t$ with their labelled symbols so that the source of $\sigma_1^l$ is $s^l$;
		\item Apply the previous point to $\sigma_{i+1}$ such that the source of $\sigma_{i+1}^l$ coincides with the target of $\sigma_i^l$;
		\item Repeat the step for every $i$.
	\end{itemize}
\end{definition}
\begin{example}
	Consider $\sigma_1 = \id{a} \tensor \sigma_{a,a} \tensor \id{b \tensor b}$ and $\sigma_2 = \id{a \tensor a \tensor a} \tensor \sigma_{b,b}$. Then $\sigma^l_1;\sigma^l_2$ is given by composing $\sigma^l_1 = \id{a^1} \tensor \sigma_{a^2, a^3} \tensor \id{b^1 \tensor b^2}$ and $\sigma^l_2 = \id{a^1 \tensor a^3 \tensor a^2} \tensor \sigma_{b^1, b^2}$.
\end{example}
Note how Definitions~\ref{def: lifting strings} and~\ref{def: lifting symmetries} are well-defined since every object in a free symmetric monoidal category can be decomposed in a unique way.
\begin{definition}
	A symmetry $\sigma: s \to s'$ is called \emph{swap-free} if, for \emph{any} decomposition of $\sigma$ into basic blocks $\sigma = \sigma_1; \dots; \sigma_n$ and fixed labelings of generators $t^j, t^k$, there is a bijective correspondence between the sets:
	\begin{equation*}
		\suchthat{i}{\sigma_i^l := \id{u} \tensor \sigma_{t^j,t^k} \tensor \id{t}} \simeq \suchthat{i}{\sigma_i^l := \id{u'} \tensor \sigma_{t^k,t^j} \tensor \id{t'}}
	\end{equation*}
\end{definition}
Intuitively, a symmetry is swap-free if crossing wires with the same label can always be disentangled to identities. Indeed, it is easy to check that symmetries between objects which are products of distinct generators are always swap-free. Moreover, swap-free symmetries are well-behaved:
\begin{proposition}\label{prop: swap-free compose}
	Composition of swap-free symmetries is swap-free.
\end{proposition}
\begin{proof}
	Consider $\sigma, \tau$ composable and swap-free. Then, let $\sigma_1;\dots;\sigma_n;\tau_1;\dots;\tau_m$ be a decomposition of $\sigma;\tau$ into basic blocks. Consider $\sigma^l_1;\dots;\sigma^l_n;\tau^l_1;\dots;\tau^l_m$, and fix two labeled generators $t^j, t^l$. Then we have, applying some tedious but straightforward index rewriting to the $\tau$s:
	\begin{align*}
		&\suchthat{i}{\sigma_i^l := \id{u} \tensor \sigma_{t^j,t^k} \tensor \id{t}} 
		\sqcup 
		\suchthat{i}{\tau_i^l := \id{u} \tensor \tau_{t^j,t^k} \tensor \id{t}} 
		\simeq\\ 
		&\qquad\qquad \simeq \suchthat{i}{\sigma_i^l := \id{u'} \tensor \sigma_{t^k,t^j} \tensor \id{t'}} 
		\sqcup
		\suchthat{i}{\tau_i^l := \id{u} \tensor \tau_{t^j,t^k} \tensor \id{t}}\\
	\end{align*}
	Where the bijection follows because the sets are by hypothesis pairwise bijective. 
\end{proof}
On the contrary, composition of non-swap-free symmetries is not preserved, as we can see in the following example:
\begin{example}\label{ex: non-swap-free do not compose}
	Consider the symmetry $\sigma_{a,a}$. It is clearly not swap-free, but it can be composed with itself, giving $\sigma_{a,a};\sigma_{a,a} = \id{a \tensor a}$ which is swap-free.
\end{example}
\subsection{Goodbye, functoriality!}
The relevance of swap-free symmetries relies on the fact that ``counting how many times wires cross each other'' is the only property that, modulo coherence conditions for SMCs, allows us to distinguish different symmetries between objects. 

In fact, if a FSSMC $\CategoryC$ has objects generated by a set $S$, then we can faithfully map $\Sym{S}$ to $\CategoryC$, and all the considerations we made on symmetries of $\Sym{S}$ can be transported to $\CategoryC$. By generalizing Proposition~\ref{prop: different objects unique symmetry}, given two objects there is at most one symmetry between them once we decide how -- if at all -- wires labeled in the same way have to cross. In particular, there is at most one swap-free symmetry between any two objects.

Proposition~\ref{prop: swap-free compose} and Example~\ref{ex: non-swap-free do not compose} determine how swap-freeness is the only type of wire-crossing which is invariant under composition, and hence the only viable choice for $\bar{\sigma}$ and $\bar{\sigma}'$ in Definition~\ref{def: free functor on morphisms}. Unfortunately, this is when our valiant chase after adjunctions collides with reality: Our choice is not functorial! In fact:
\begin{theorem}\label{prop: swap-freeness is not functorial}
	A transition-preserving functor preserves swap-free symmetries iff it is injective on objects.
\end{theorem}
\begin{proof}
	We prove the counternominal, which is easier. Suppose $F: \CategoryC \to \CategoryD$ is not injective on objects. Then there are generating objects $s_1 \neq s_2$ of $\CategoryC$ such that $Fs_1 = Fs_2$. Hence $\sigma_{s_1,s_2}$ is swap-free but $F\sigma_{s_1 s_2} = \sigma_{Fs_1, Fs_2} = \sigma_{Fs_1, Fs_1}$ is not, and $F$ doesn't preserve swap-free symmetries.
	
	Conversely, suppose that $F$ is injective on objects. Then if $\sigma$ is swap-free and applying Definition~\ref{def: lifting strings} we get, for any decomposition of $\sigma$ into basic blocks and each couple of labelled generators $t^j, t^k$, an isomorphism
	\begin{equation*}
		\suchthat{i}{\sigma_i^l := \id{u} \tensor \sigma_{t^j,t^k} \tensor \id{t}} \simeq \suchthat{i}{\sigma_i^l := \id{u'} \tensor \sigma_{t^k,t^j} \tensor \id{t'}}
	\end{equation*}
	Since $F$ is injective, all the objects in $\CategoryC$ are mapped to different objects in $\CategoryD$, and labelings are preserved just by rewriting generators, so we get $F(t^j) = (Ft)^j$, $F(t^k) = (Ft)^k$. From this we can extend the isomorphism above to 
	\begin{equation*}
		\suchthat{i}{F\sigma_i^l := \id{Fu} \tensor \sigma_{(Ft)^j,(Ft)^k} \tensor \id{Ft}} \simeq \suchthat{i}{\sigma_i^l := \id{Fu'} \tensor \sigma_{(Ft)^k,(Ft)^j} \tensor \id{Ft'}}
	\end{equation*}
	This is sufficient to prove swap-freeness of $F\sigma$, since all generating objects in $F\sigma$ are the image of generating objects of $\sigma$ through $F$. 
\end{proof}
\begin{example}\label{ex: functoriality not working}
	Suppose we require $\bar{\sigma}$ and $\bar{\sigma}'$ to be swap-free in Definition~\ref{def: free functor on morphisms}. Consider net morphisms $N \xrightarrow{f} M \xrightarrow{g} L$. By definition, if $f(t_N) = t_M$, then $\OFoldSym f$ will send the generator $t_N$ to $\sigma;t_M;\sigma'$, with $\sigma, \sigma'$ swap-free symmetries in $\OFoldSym M$. Again by definition, if $g(t_M) = t_L$, $\OFoldSym g$ will send $t_M$ to $\tau; t_L; \tau'$, with $\tau, \tau'$ swap-free in $\OFoldSym L$. Then the composition $\OFoldSym f ; \OFoldSym g$ will send $t_N$ to:
	\begin{equation*}
		(\OFoldSym g) \sigma; (\OFoldSym g) t_M; (\OFoldSym g) \sigma' = (\OFoldSym g) \sigma; \tau; t_L; \tau' ; (\OFoldSym g) \sigma'
	\end{equation*}
	If $g$ is not injective on places, $(\OFoldSym g) \sigma$, $(\OFoldSym g) \sigma'$ in are not in general swap-free. Hence, we cannot conclude that $(\OFoldSym g) \sigma; \tau$ and $\tau';(\OFoldSym g) \sigma'$ are swap-free. On the other hand, $\OFoldSym(f;g)$ sends $t_N$ to $\rho;t_L;\rho'$ with $\rho, \rho'$ swap-free in $\OFoldSym L$, proving $\OFoldSym(f;g) \neq \OFoldSym f ; \OFoldSym g$. 
\end{example}
The significance of Proposition~\ref{prop: swap-freeness is not functorial}, in light of Example~\ref{ex: functoriality not working}, is that if we want our requirements in Definition~\ref{def: requirements} to be satisfied we have to give up functoriality: A functor between SSMCs is injective on objects only if its corresponding morphism between nets is, and restricting to injective net morphisms is unacceptable for our requirements. Hence, we conclude that there is no way to isolate ``nice'' ways to linearize multisets homomorphisms just from the topological properties of symmetries, at least not without violating Requirement~\ref{item: morphisms are automatically lifted} in our list.

Note also that we cannot hope for a non-invasive quick fix as we did in Definition~\ref{def: free functor on objects}, since in linearizing a multiset homomorphism we are forced to consider the relationship between its source and target: The only way to lift a net morphism to a functor between FSSMCs is to impose additional structure on the net \emph{which the morphisms are required to respect}. This is akin to the strategy adopted for pre-nets, which we have already deemed unsuitable for our purposes.
\subsection{A word on higher-categorical approaches}
One may be tempted to investigate a higher-categorical correspondence between nets and FSSMCs, replacing functors with pseudofunctors, and chasing a weak 2-adjunction instead. This approach does not work either. First, notice how natural transformations, which are the standard way to define 2-cells in our situation, do not really help: A natural transformation between functors $F,G: \CategoryC \to \CategoryD$ acts by selecting, for each object of $\CategoryC$, a morphism of $\CategoryD$. On the contrary, a transition-preserving functor acts by selecting a couple of symmetries in $\CategoryD$ for \emph{each} morphism generator in $\CategoryC$. This is to say that the idea of defining a natural transformation whose components are all symmetries between $\OFoldSym(f;g)$ and $\OFoldSym f ; \OFoldSym g$ won't work, since differences between the two are too granular to satisfy the required naturality conditions.
 
Another approach may be replacing natural transformations with endofunctors sending generators to permutations of themselves, and use such endofunctors to define 2-cells. This is the strategy that the authors pursued in trying both to preserve the adjunction and to have their requirements satisfied, before realizing this was not possible. Indeed, careful investigation of the categorical structure shows how, for these 2-cells to be effective in defining a pseudofunctorial correspondence between nets and FSSMCs, one needs to restrict to transition-preserving functors which are faithful, again violating Requirement~\ref{item: morphisms are automatically lifted}.

As a final note, the authors want to stress that implementing a 1-categorical version of the correspondence between nets and FSSMCs is already pushing the Idris compiler to its limits~\cite{Perone2019}. Hence, even supposing that it exists, the absence of tactics and other refined theorem-proving tools in all production-ready programming languages would indeed make implementing a 2-categorical version of our problem very difficult. Knowing this, the authors kept their investigations in the realm of higher category theory as limited as possible.
\section{A different approach}
Although the last section may seem to cast a somewhat pessimistic light on the question of how to relate nets and FSSMCs in a implementation-friendly way, all is not lost. In fact, we have the following result:
\begin{proposition}
	There is a functor from $\FSSMC$ to $\Petri$, which can be restricted to a functor from $\GFSSMC$ to $\GPetri$.
\end{proposition}
\begin{proof}
	We send each FSSMC $\CategoryC$ to $\UnFoldSym \CategoryC$. If $F: \CategoryC \to \CategoryD$ is a transition-preserving functor sending $t_\CategoryC$ to $\sigma;t_\CategoryD;\sigma'$, we define $\UnFoldSym F: \UnFoldSym \CategoryC \to \UnFoldSym \CategoryD$ by sending the transition $t_\CategoryC$ to the transition $t_\CategoryD$. On places, we just set $\UnFoldSym F p = \Multiplicity{} F p$. Functoriality is a straightforward check.
	
	Restricting to grounded FSSMCs, if $F:\CategoryC \to \CategoryD$ is grounded then for each generating object $s$ in $\CategoryC$ $F s$ is a generating object in $\CategoryD$. By definition, so is $\Multiplicity{} F s$, straightforwardly proving the claim. 
\end{proof}
As expected, $F$ acts by forgetting information, effectively delinearizing mappings of strings to multiset homomorphisms. 
\begin{figure}[!h]
	\centering
		\scalebox{0.6}{
		\begin{tikzpicture}
			\node (c) {$\xleftarrow{\UnFoldSym}$};
			\node (1a) [above left = 1cm and 1cm of c]{$N$};
			\node (2a) [above right = 1cm and 1cm of c] {$\CategoryC$};
			\node (1b) [below = 2.7cm of 1a] {$M$};
			\node (2b) [below = 2.7cm of 2a] {$\CategoryD$};
			
			\draw[->, bend left] (1a) to node[midway, above] {$\OFoldSym$}  (2a);
			\draw[->, bend left] (1b) to node[midway, above] {$\OFoldSym$}  (2b);
			\draw[<-, bend right] (1a) to node[midway, above] {$\UnFoldSym$}  (2a);
			\draw[<-, bend right] (1b) to node[midway, above] {$\UnFoldSym$}  (2b);		
			
			\draw[->] (1a) to node[midway, left] {$f$}  (1b);
			\draw[->] (2a) to node[midway, right] {$F$}  (2b);	
		\end{tikzpicture}
		}
	\caption{The mappings between nets and FSSMCs.}\label{fig: Petri-FSSMC mapping}
\end{figure}
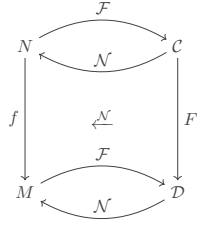

\noindent
Putting everything together, the situation is as in Figure~\ref{fig: Petri-FSSMC mapping}. There is a way to go \emph{back and forth} from nets to FSSMCs, and a way to functorially map FSSMC morphisms to net morphisms. This suggests the following course of action: In our implementation, we \emph{always} manipulate FSSMCs. The user can nevertheless draw Petri nets, which are automatically mapped to their corresponding FSSMC through $\OFoldSym$. When the user wants to fetch information regarding some Petri net $N$, the corresponding FSSMC $\OFoldSym N$ is retrieved instead, and $\UnFoldSym \OFoldSym N$ is displayed. 

Moreover, the user can specify a morphism between nets $f: N \to M$ by selecting how places and transitions are to be mapped. By applying Proposition~\ref{def: free functor on morphisms}, we can lift this information (\emph{non-functorially!}) to a functor $\OFoldSym N \to \OFoldSym M$ which maps $t_\CategoryC$ to $\sigma;t_\CategoryD;\sigma'$ where $\TrM{f}(t_\CategoryC) = t_\CategoryD$ and $\sigma, \sigma'$ are swap-free. 

Notice that a transition-preserving functor $F: \CategoryC \to \CategoryD$ can be fully specified by providing a mapping between objects, a mapping between the set of generators of $\CategoryC$ and $\CategoryD$ and, for each $t_\CategoryC$ generator of $\CategoryC$, a couple $(\sigma, \sigma')$ of suitable symmetries in $\CategoryD$. The strategy highlighted in the paragraph above provides a way to automatically infer suitable couples $(\sigma, \sigma')$ for each generator $t_\CategoryC$ by providing a net morphism. \emph{In our setting, this can be seen as analogous to type inference}. Having obtained this list, the user can then further tweak each couple to suit particular needs. All in all, the user \emph{never} has to specify a morphism of nets, and instead works with FSSMCs from the start. However, morphisms between nets \emph{can} be used to automatically provide a template from which a morphism between their corresponding FSSMCs is refined and specified.
\subsection{Implementation advantages}
Notice how the strategy summarized respects all requirements of Definition~\ref{def: requirements}: 
\begin{itemize}
	\item All the categorical structures involved are free. Moreover, we are able to employ the definition of Petri net and FSSMC  naively, with no additional structure curbing the expressiveness of both paradigms;
	\item Mapping FSSMCs to Petri nets is as computationally feasible as just implementing FSSMCs, since nearly all the operations between nets and FSSMCs amount to forget some information;
	\item FFSMCs represent net computations in a meaningful way: Each morphism in the FSSMCs can be mapped to a sequence of transition firings in the corresponding net, and for each sequence of transition firings there is at least a corresponding morphism in the FSSMC. Moreover, each FSSMCs can be functorially mapped to any other symmetric monoidal category, providing the functorial mapping to a semantics as we wanted;
	\item Our mapping is useful for intuition, since we are able to naively leverage the graphical calculus of Petri nets on one hand, and string diagrams to represent net executions on the other;
	\item Nets can be morphed into each other. Moreover, the user is able to tweak the morphism between corresponding net histories directly by specifying a transition-preserving functor between executions.
\end{itemize}
Another great advantage of using FSSMCs directly is that \emph{manipulating strings is way easier than manipuating multisets in a development environment}. This is because countless tools and data structures, such as lists, have been developed to deal with strings over the years, mainly to perform operations on text. To see the extent to which this is true, consider the method to parse the information defining a net from a string shown in Figure~\ref{fig: string to net}:
\begin{figure}[!h]
	\centering
	\includegraphics[height=100px]{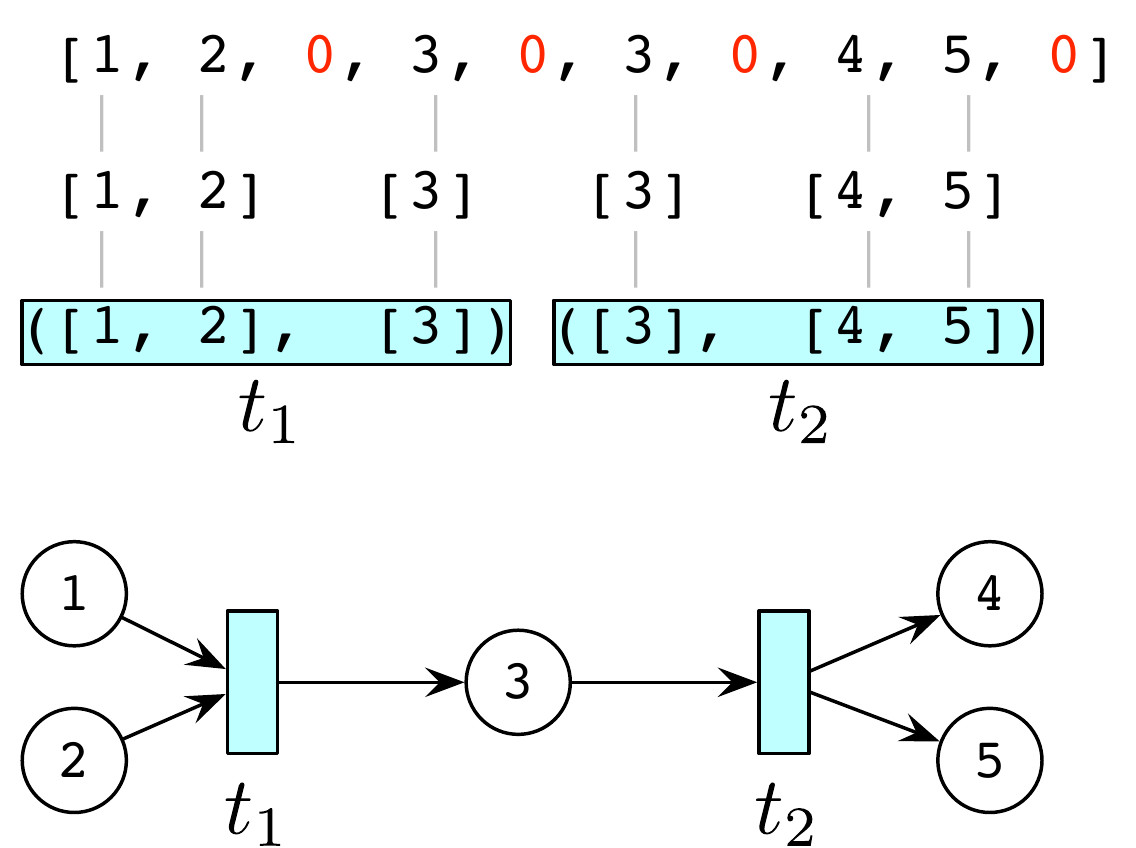}
	\caption{A way to convert a string to a Petri net, and vice-versa.}\label{fig: string to net}
\end{figure}

\noindent
We start with a list of natural numbers, where each number uniquely labels a place in the net. The number $0$ is special, and is used to split the list into sublists. Starting from the left, sublists are paired in couples, each couple defining input and output places of a transition. As the figure shows, this is enough information to construct a Petri net.

Such a procedure is fairly straightforward, but we might notice how such a list could also be used to define the symmetric monoidal category corresponding to the net. Lists are in fact naturally ordered, and each sublist in a pair can be used to define the source and target of a generating morphism: All in all, programming languages \emph{do prefer} working with FSSMCs, which is where many of the computational advantages of our strategy come from.
\subsection{Status of the implementation}
Having focused on implementation throughout this document, the reader may have begun to wonder whether such an implementation exists. Happily, this is indeed the case: Quite recently, we released \idrisct~\cite{StateboxTeamb}, an Idris library that provides formally verified definitions of categorical concepts such as category, functor, natural transformation etc. In particular, with \idrisct it is already possible to build symmetric monoidal categories specifying a list of generators in the very same format shown in Figure~\ref{fig: string to net}. 

The functorial mapping from $\FSSMC$ to $\Petri$ has been worked out on a separate repository, which will be made public soon. The user interface code for visualizing nets and their executions in the frontend is also being worked on. The most challenging part of that is interconnecting the Idris code developed in \idrisct with more flexible languages, such as Purescript~\cite{PureScript}, used in the frontend and middleware.

Another direction of development consists in functorially linking \idrisct with \texttt{typedefs}~\cite{StateboxTeama}, a language-agnostic type construction language based on polynomials. Such a functor will enable the mapping of net executions to computations, thus fully realising Requirement~\ref{item: mapping is faithful and full}.
\section{Conclusion and future work}
In this paper, we reviewed some of the approaches pursued in linking Petri nets with free strict symmetric monoidal categories, traditionally interpreted as their categories of computations. After fixing a list of requirements, we argued that chasing an adjunction between the two may not be the most fruitful strategy if one's goal is to implement such a correspondence in code. As an alternative, we showed how a functor from free strict symmetric monoidal categories to Petri nets is enough to obtain a conceptually meaningful implementation that fully satisfies our requirements.

We consider this important since such an implementation has immediate practical implications, in that it enables the direct use of Petri nets in software development.

Future work will be focused mainly on fully implementing the strategy highlighted in this paper. On the purely categorical side, we would like to investigate whether the results recently published in~\cite{Balco2019} have any impact in representing Petri net computations.

\printbibliography
\newpage
\appendix
\section{Appendix}\label{app: Appendix}
\subsection{Ordered Petri nets}
Here we show how we can endow Petri nets with some additional structure which allows to define a function $\Ordering{}$ to linearize multisets to strings. We will also show how this function ``does not get in the way'', meaning that it does not require to change the definition of morphism between nets. We start by observing that defining $\Ordering{}$ is sraightforward if our set of places is ordered:
%
%
%
\begin{definition}
Let $(S, <)$ be a well ordered set. Define $\Ordering{S}: \Strings{S} \to \Strings{S}$ as the function such that:
	\begin{itemize}
		\item $\Ordering{S}(s)$ is a permutation of $s$;
		\item Elements of $\Ordering{S}(s)$ are monotonically increasing: If $\Ordering{S}(s) = (s_0, \dots, s_n)$, then $s_i \leq s_{i+1}$ for each $0 \leq i < n$.
	\end{itemize}
	Abusing notation, we can define $\Ordering{S}: \Msets{S} \to \Strings{S}$ as the function mapping each multiset $s$ to the unique string $\Ordering{S}(s)$ such that $\Multiplicity{S}(\Ordering{S}(s)) = s$ and $\Ordering{S}(\Ordering{S}(s)) = \Ordering{S}(s)$.
\end{definition}
In our context, we required the set of places of a net to have no additional structure. It seems sensible then to give the following definition:
\begin{definition}\label{def: ordered Petri net}
	A \emph{ordered Petri net} is a couple $(N, <_N)$ where $N$ is a Petri net and $<_N$ is a well ordering on $\Pl{N}$.
\end{definition}
Since our main focus is on implementation, it should be noted how, given a Petri net, imposing a total ordering on its places is an easy task. For instance, in the following Idris code snippet, we define the type \texttt{OrdPetriNet} by just taking a previously defined \texttt{PetriNet} type, and by requiring that the type of places we use in it must be orderable:
\lstset{language=literateidris, style=numbers}
\begin{lstlisting}
 record OrdPetriNet where
   constructor MkOrdPetriNet
   pnet : PetriNet
   ordPlaces: Ord (places pnet)
\end{lstlisting}
Still, we need to define what a morphism of ordered Petri nets is. This is quite straightforward:
\begin{definition}
	A \emph{morphism of ordered Petri nets} $f:(N, <_N) \to (M, <_M)$ is just a morphism $N \to M$ between their underlying nets.
\end{definition}
We make no requirement whatsoever for the ordering to interact with the functions defining our morphisms. This is the main difference with the notion of pre-net defined in~\cite{Baldan2003}, where morphisms are required to respect the order in which inputs/outputs are attached to transitions. On the contrary, our definition does not violate Requirement~\ref{item: mapping is useful for intuition}. We can make this concept categorically precise by proving the following couple of propositions:
\begin{proposition}
	Ordered Petri nets and their morphisms form a category, denoted $\OPetri$.
\end{proposition}
\begin{proof}
	Trivial by noticing that, for each net $(N, <_N)$, $\id{N} := (\id{\Msets{\Pl{N}}}, \id{\Tr{N}})$ is a morphism $(N, <_N) \to (N, <_N)$ that respects the identity laws. Morphism composition is defined in the obvious way from composition of multiset homomorphisms and functions, as does associativity. 
\end{proof}
\begin{proposition}\label{prop: ordered and unordered Petri are equivalent}
	$\OPetri$ is equivalent to $\Petri$, the category of Petri nets and their morphisms.
\end{proposition}
\begin{proof}
	This is obvious by noting that we can send any Petri net to an ordered net by picking any ordering procedure on the set of its places.\footnote{This is guaranteed in the general case by assuming the well-ordering theorem. Such claim is not problematic considering that the well-ordering theorem can be realized constructively if the set of places is finite, which is a necessary assumption in implementations.} Since ordering does not play any role in the definition of ordered nets morphisms, such mapping can be extended to a functor which is identity on morphisms, hence fully faithful. Moreover, if two ordered nets have the same underlying net, as in$(M, <)$ and $(M, <')$ for some Petri net $M$, then they are isomorphic in $\OPetri$ via the morphism $(\id{\Msets{\Pl{M}}}, \id{\Tr{M}})$. This proves that our functor is also essentially surjective, concluding the proof. 
\end{proof}
Unsurprisingly, grounded morphisms behave smoothly with respect to the imposed ordering structure:
\begin{proposition}
	Petri nets and grounded morphisms between them form a subcategory of $\OPetri$, denoted $\GPetri_<$. $\GPetri_<$ is equivalent to $\GPetri$.
\end{proposition}
\subsection*{Ordered FFSMCs}
Similarly to what we did with Petri nets, we can impose an order structure on FSSMCs. Everything is conceptually similar to what we did in the last section.
\begin{definition}
	An \emph{ordered FSSMC} is a couple $(\CategoryC, <_\CategoryC)$ where $\CategoryC$ is a FSSMC and $<_\CategoryC$ is a well order on its set of object generators. Ordered FSSMCs and transition-preserving functors between them form a category, denoted $\OFSSMC$.
\end{definition}
As for Definition~\ref{def: ordered Petri net}, the ordering on generating objects goes basically undetected by the morphism structure. Hence we can swiftly conclude that:
\begin{fact}
	$\OFSSMC$ is equivalent to $\FSSMC$. Ordered FSSMCs and grounded transition-preserving functors form a subcategory of $\OFSSMC$, denoted $\GFSSMC_<$. $\GFSSMC_<$ is equivalent to $\GFSSMC_<$.
\end{fact}
\subsection*{Redefining the mappings}
Since we added extra structure both to nets and FSSMCs, we need to slightly refine the mappings we already defined in Definitions~\ref{def: free functor on objects} and~\ref{def: forgetful functor on objects}:
\begin{definition}
	Given an ordered Petri net $(N, <_N)$, we map it to the ordered, free-on-objects strict symmetric monoidal category $(\OFoldSym N, <_{\OFoldSym N})$, where the order relations $<_N$ and $<_{\OFoldSym N}$ are the same.
\end{definition}
\begin{definition}
	To an ordered FSSMC $(\CategoryC, <_\CategoryC)$ we associate the net $(\UnFoldSym \CategoryC, <_{\UnFoldSym \CategoryC})$, where the order relations $<_\CategoryC$ and $<_{\UnFoldSym \CategoryC}$ are the same.
\end{definition}
Since the ordering doesn't play any role in the definition of morphisms both on the nets side and on the FSSMCs side, all the results proven in the paper can be generalized to the case of ordered nets/FSSMCs with minimal modifications.
\end{document}